\documentclass[11pt,reqno]{amsart}
\usepackage[usenames]{color} 
\usepackage{amssymb} 
\usepackage{amsmath} 
\usepackage{verbatim}
\usepackage{}
\usepackage{mathrsfs}
\usepackage{amsfonts}
\usepackage{cite}
\usepackage{amssymb,bm}
\usepackage[utf8]{inputenc} 
\DeclareMathOperator{\curl}{curl}

\DeclareMathOperator{\dv}{div}

\date{\today}

\def\subjclass#1{{\renewcommand{\thefootnote}{}%
\footnote{\emph{Mathematics Subject Classification (2010):} #1}}}

\hoffset=- 2cm \voffset=- 0cm 
 \theoremstyle{plain}
\newtheorem{Thm}{Theorem}

\newtheorem{Def}[Thm]{Definition}

\usepackage{amssymb}
\usepackage[active]{srcltx}
\usepackage{color}

\newcommand {\p}{\partial}
\newcommand{\q}{\quad}

\def\O{\Omega}

\def\A{\mathbf A}
\def\B{\mathbf B}

\def\u{\mathbf u}

\def\v{\vskip}

\errorcontextlines=0 \numberwithin{equation}{section}
\numberwithin{Thm}{section}

\begin{document}
\large

\title{Global boundedness of the curl for a $p$-curl
 system in convex domains}

\author[]{Hongjin Wu, Baojun Bian}

\address{School of Mathematical Sciences,
Tongji University, Shanghai 200092, P.R. China}
\email{wuhongjin@tongji.edu.cn}

\thanks{ }

\keywords{$p$-curl systems, convex domains, $L^{\infty}$ estimate}

\subjclass{26D10; 46E40; 35Q61; 82D55}
\begin{abstract}
In this paper, we study a semilinear system involving the curl operator in a bounded and convex domain in
$\mathbb{R}^3$, which comes from the steady-state approximation for Bean  critical-state model for type-$\mathrm{II}$  superconductors. We show the existence and the $L^{\infty}$ estimate for weak solutions to this system.
\end{abstract}
\maketitle

\section{Introduction}\label{section1}
This paper is devoted to the study of the boundedness of the curl of solutions to a semilinear system in a bounded and convex domain $\O$ in $\mathbb{R}^3$:
\begin{equation}\label{1.1}
\begin{cases}
\curl\left(|\curl\u|^{p-2}\curl\u\right)={\bm f} &\text{\rm in } \O,\\
\dv\u=0 &\text{\rm in } \O,\\
\u\times \bm\nu =0 &\text{\rm on } \p \O,
\end{cases}
\end{equation}
where $\bm\nu$ denotes the outward unit normal to $\p \O$, $p\in (1,\infty)$ is fixed, $\mathbf{u}: \O \to \mathbb{R}^3$ is a vector-valued unknown function and ${\bm f} :\O \to \mathbb{R}^3$ is a given and divergence-free vector.

The system \eqref{1.1} appears as a model of the magnetic induction in a high-temperature superconductor operating near its critical current (see \cite{Yin2006Regularity, Laforest}). More precisely, $\u$ and $\curl\u$ are the magnetic field and the total current density respectively, and ${\bm f}$ denotes the internal magnetic current. The system \eqref{1.1} is also the steady-state approximation for Bean's critical-state model for type-$\mathrm{II}$ superconductors (see \cite{Chapman2000A,Yin2001,Yin2002}), where the magnetic field $\u$ is approximated by the solution of the $p$-curl-evolution system. For the detailed physical background, we refer to \cite{Bean1964Magnetization, Chapman2000A, DeGennes}.

We need to mention that the Cauchy problem for the $p$-curl-evolution system in unbounded domains can be found in \cite{Yin2001} by Yin, who studied the existence, uniqueness, and regularity of solutions to this system. In \cite{Yin2002}, Yin considered the same problem of the $p$-curl-evolution system in $\Omega\times(0,T]$, where $\Omega$ is a bounded domain in $\mathbf{R}^3$ with $C^{1,1}$ boundary and no holes in the interior. The results for domains in $\mathbf{R}^2$ had been established in \cite{Barrett}. Recently, the work of \cite{Yin2002} had been extended to more general resistivity term of the form $\rho=g(x,|\curl\mathbf{H}|)$ for some function $g$ by Aramaki (see \cite{Aramaki}). We refer to \cite{Antontsev} and \cite{FAntontsev} for more results of the parabolic version of the system \eqref{1.1}. Some recent results with respect to this kind of problems can be found in \cite{Bartsch, Mederski,Tang}.

Regardless of the time dimension, Laforest \cite{Laforest} studied the existence and uniqueness of the weak solution to the steady-state $p$-curl system \eqref{1.1} in a bounded domain whose boundary is sufficiently smooth. Yin\cite{Yin2006Regularity} studied the regularity of the weak solution to the system \eqref{1.1} in a bounded and simply-connected domain $\Omega$ with $\p\O\in C^2$ and showed that the optimal regularity weak solutions is of class $C^{1+\alpha}$, where $\alpha\in(0,1)$. For such problems mentioned above, one usually works in spaces of divergence-free vector fields. In \cite{ChenJun}, Chen and Pan considered the existence of solutions to a quasilinear degenerate elliptic system with lower-order terms in a bounded domain under the Dirichlet boundary condition or the Neumann boundary
condition, and the divergence-free subspaces are no longer suitable as admissible spaces. We also mention that further generalizations of $p$-curl systems with $p$ instead of $p(x)$ are the subject of \cite{FAntontsev, XiangM} and the references therein.

Moreover, system \eqref{1.1} is a natural generalization of the $p$-Laplacian system for which a well-developed theory already exists (see \cite{Cianchi2014Global, Barrett1993Finite, Dibenedetto2010Degenerate, Yin1998On, Uralceva}). In particular, the boundedness of the gradient up to the boundary for solutions to Dirichlet and Neumann problems for  divergence form  elliptic systems with Uhlenbeck type structure was established by Cianchi and Maz'ya in \cite{Cianchi2014Global}, where the $p$-Laplacian system is also included. Note that the system \eqref{1.1} is a curl-type elliptic system with Uhlenbeck type structure. In parallel with the problem for $p$-Laplacian systems in \cite{Cianchi2014Global}, our study is aimed at checking whether we have the boundedness of the curl of solutions to the system \eqref{1.1}. The elliptic characteristic of system \eqref{1.1} can be inferred by this result and the system thus can satisfy the properties of elliptic partial differential equations. But unlike the Dirichlet boundary condition or the Neumann boundary condition for the divergence-type system in \cite{Cianchi2014Global}, the difficulties of our work is to tackle the vanished tangential components and the degeneration of the operator $curlcurl$ in system \eqref{1.1}. 

To state our result, we first need to introduce Sobolev spaces $W_t^p(\O,\dv 0)$ with $1<p<\infty$
and Lorentz spaces $L(m,p)$ with $1\leq m<\infty$ and $1\leq p\leq\infty$.
Define by
$$
W_t^p(\O,\dv 0):=\left\{\u, \curl\u\in L^{p}(\O)~:~ \dv\u=0 \text{ in } \O
\text{ and }\bm\nu\times\u=0 \text{ on }\p\O\right\}
$$
with the norm
$$
\|\u\|_{W_t^p(\O,\dv 0)} = \|\u\|_{L^p(\O)}+\|\curl\u\|_{L^p(\O)}.
$$

Let $(X,S,\mu)$ be a $\sigma-$finite measure space and $f:X\to\mathbb{R}$ be
a measurable function. We define the distribution function of $f$ as
$$f_{*}(s)=\mu(\{|f|>s\}),\q s>0,
$$
and the non-increasing rearrangement of $f$ as
\begin{equation}\label{7.2}
f^{*}(t)=\inf\{s>0, f_{*}(s)\leq t\},\q t>0.
\end{equation}
The Lorentz space now is defined by
$$L(m,p)=\left\{f: X\to\mathbb{R} \text{ measurable},
\|f\|_{L^{m,p}}<\infty\right\} \q \text{with } 1\leq m<\infty
$$
equipped with the quasi-norm
$$\|f\|_{L(m,p)}=\Big(\int_0^{\infty}\left(t^{1/m}f^{*}(t)\right)^p\frac{dt}{t}\Big)^{1/p},
\q 1\leq p<\infty
$$
and
$$\|f\|_{L(m,\infty)}=\sup_{t>0} t^{1/m} f^{*}(t),\q p=\infty.
$$

Then we introduce the weak solution to the system \eqref{1.1}:
\begin{Def}
We say that $\u\in W_t^p(\O,\dv 0)$ is a (weak) solution to the system \eqref{1.1}
if
$$
\int_{\O} |\curl\u|^{p-2}\curl\u\cdot\curl\mathbf\Phi dx=\int_{\O} {\bm f}\cdot\mathbf\Phi dx
$$
for any $\mathbf\Phi\in W_t^p(\O,\dv 0).$
\end{Def}

Our result now reads as follows:
\begin{Thm}\label{Theorem1.1}
Let $\O$ be a bounded convex domain in $\mathbb{R}^3$.
Assume that ${\bm f} \in L^{3,1}(\O)$ with $\dv{\bm f}=0$ in the sense of distribution, then there exists a unique (weak) solution $\u\in W_t^p(\O,\dv 0)$ to the system \eqref{1.1}. Moreover, we have the estimate
\begin{equation}\label{1.2}
\lVert\curl\u\rVert _{L^{\infty}(\O)}\leq C \|{\bm f}\|_{L^{3,1}(\O)}^{\frac{1}{p-1}},
\end{equation}
where the constant $C$ depends on $p$ and $\O.$
\end{Thm}

The proof of Theorem \ref{Theorem1.1} will be given in Section 2. Throughout this paper, a bold letter represents a three-dimensional vector or vector function.

\section{proof of the main result}\label{section2}

We now give the proof of our main theorem.

\begin{proof}[Proof of Theorem \ref{Theorem1.1}]
Note that the system \eqref{1.1} is the Euler equation
 of the strictly convex
functional
$$
J(\u)=\int_{\O}\left(\frac{1}{p}|\curl\u|^p -{\bm f}\cdot\u \right)dx
$$
in the space $W^p_t(\O, \dv 0).$ It is easy to see that the functional
is of weak lower semi-continuity and coercivity,
then the existence and the uniqueness of the
weak solution to the system \eqref{1.1} follows,
see \cite[Theorem 1.6]{Str}.
Since $
J(0)\leq 0,
$  then  we  have the estimate
$$
\frac{1}{p}\int_{\O}|\curl\u|^p dx\leq \|{\bm f}\|_{L^{3,1}(\O)}\|\u\|_{L^{3/2,\infty}(\O)}.
$$
The space $W^p_t(\O, \dv 0)$
is continuously embedded into the space $L^{3/2, \infty}(\O),$ we can obtain that
\begin{equation}\label{7.3}
\|\curl\u\|_{L^p(\O)}\leq C(p,\O) \|{\bm f}\|_{L^{3,1}(\O)}^{\frac{1}{p-1}}.
\end{equation}
Then, we show the inequality \eqref{1.2}. The proof is divided into three steps.

Step 1.
First, we make the following assumptions:
\begin{equation}\label{2.3}
\p\O\in C^{\infty},
\end{equation}
and
\begin{equation}\label{2.4}
{\bm f}\in C^{\infty}(\bar{\O}).
\end{equation}

From the main theorem in \cite{Yin1998On}, we see that $\u\in C^{1,\alpha}(\bar{\O})$ for some $\alpha \in (0,1).$
Then classical regularity results can show that the solution
$$\u \in C^3\left(\bar{\O}\bigcap \overline{\{|{\curl\u}|>t\}}\right),$$
for $t\geq|\bm{\omega}|^*(|\Omega|/2),$ where
$|\bm{\omega}|^*$ is defined by \eqref{7.2},

For simplicity, we now rewrite $\bm{\omega}$, $G(|\bm{\omega}|)$ instead of $\curl\u$, $|\curl\u|^{p-2}$, respectively.
We first establish the following inequality,
\begin{equation}\label{1.5}
\aligned
(p-1)t^{p-1}(\int_{\{|\bm{\omega}|=t\}}|\nabla|\bm\omega||dS)\leq &
\int_{\{|\bm{\omega}|>t\}} \frac{1}{4(p-1)}\frac{1}{|\bm\omega|^{p-2}}|{\bm f}|^2 dx
+t\int_{\{|\bm{\omega}|=t\}}|{\bm f}| dS\\
&+\int_{\p\O\bigcap \p\{|\bm{\omega}|>t\}} G(|\bm{\omega}|)
\sum_{i,j=1}^3 \nu_i \omega_j \p_j \omega_i dS.
\endaligned
\end{equation}

First, by applying the formula $\dv(\A\times\B)=\curl\A\cdot\B-\A\cdot\curl\B$ to the system \eqref{1.1}, we see that
\begin{equation}\label{1.6}
-{\bm f}\curl\bm\omega=\dv(G(|\bm\omega|)\bm\omega\times\curl\bm\omega)+G(|\bm\omega|)\bm\omega\cdot\curl\curl\bm\omega.
\end{equation}
By simple computations, it follows that
$$
\aligned
G(|\bm\omega|)\bm\omega\cdot&\curl\curl\bm\omega
=-\dv(G(|\bm\omega|)\nabla|\bm\omega||\bm\omega|)
+G(|\bm\omega|)|\nabla|\bm\omega||^2
+G'(|\bm\omega|)|\nabla|\bm\omega||^2|\bm\omega|.
\endaligned
$$
Substituting the above equality to \eqref{1.6}, we have
\begin{equation}\label{1.7}
\aligned
-{\bm f}\curl\bm\omega =\dv(G(|\bm\omega|)\bm\omega\times\curl\bm\omega)
-\dv(G(|\bm\omega|)\nabla|\bm\omega||\bm\omega|)+(p-1)|\bm\omega|^{p-2}|\nabla|\bm\omega||^2.
\endaligned
\end{equation}
By Cauchy inequality, one has that
\begin{equation}\label{1.8}
{\bm f}\curl\bm\omega\geq -\frac{|{\bm f}|^2}{4(p-1)|\bm\omega|^{p-2}}-(p-1)|\bm\omega|^{p-2}|\curl\bm\omega|^2.
\end{equation}
Coupling \eqref{1.7} with \eqref{1.8} tells us that
\begin{displaymath}
\aligned
&-\dv(G(|\bm\omega|)\bm\omega\times\curl\bm\omega)+\dv(G(|\bm\omega|)\nabla|\bm\omega||\bm\omega|)\\
&\geq-\frac{1}{4(p-1)}\frac{1}{|\bm\omega|^{p-2}}|{\bm f}|^2.
\endaligned
\end{displaymath}
Integrating the above inequality over the region ${\{|\bm\omega|>t\}}$ and by Green's formula, we have
\begin{equation}\label{1.9}
\aligned
&\int_{\p\{|\bm{\omega}|>t\}}
\left[G(|\bm{\omega}|)\curl \bm{\omega}\times \bm{\omega} +G(|\bm{\omega}|)\nabla|\bm{\omega}||\bm{\omega}|\right]\cdot\bm\nu dS\\
&\geq
-\int_{\{|\bm{\omega}|>t\}}\frac{1}{4(p-1)}\frac{1}{|\bm\omega|^{p-2}}|{\bm f}|^2 dx,
\endaligned
\end{equation}
where $\bm\nu(x)$ represents the unit outer normal vector at
$x\in \p\{|\bm{\omega}|>t\}.$

For almost every $t>0$, The level surface $\p\left\{|\nabla\bm\omega|>t\right\}$ has
$$
\p\{|\bm{\omega}|>t\}=\p\O\bigcap\p\{|\bm{\omega}|>t\} +\{|\bm{\omega}|=t\}.
$$
Also, for $x\in \{|\bm{\omega}|=t\}\bigcap\{|\nabla |\bm{\omega}||\neq 0\}$ we have
\begin{equation}\label{1.10}
\bm\nu(x)=-\frac{\nabla |\bm{\omega}|}{|\nabla |\bm{\omega}||}.
\end{equation}
From Sard's theorem, we know that
\begin{center}
the image $|\bm{\omega}|(X)$ has Lebesgue measure $0$, where $X=\{|\nabla |\bm{\omega}||=0\}$.
\end{center}

Let us focus on the terms in the left-hand side of \eqref{1.9}. Since, for every $x \in \p\O$,
$$\left(\curl \bm{\omega}\times \bm{\omega} +\nabla|\bm{\omega}||\bm{\omega}|\right)\cdot\bm\nu(x)
=\sum_{i,j=1}^3\nu_i \omega_j \p_j \omega_i.$$
Combining \eqref{1.1} and \eqref{1.10}, and making use of the above equality yields
\begin{equation}\label{1.11}
\aligned
&\int_{\p\{|\bm{\omega}|>t\}}[G(|\bm\omega|)\curl\bm\omega\times\bm\omega+G(|\bm\omega|)\nabla|\bm\omega||\bm\omega|]\cdot\bm\nu dS\\
\leq&
\int_{\p\O\bigcap \p\{|\bm{\omega}|>t\}} G(|\bm{\omega}|)\sum_{i,j=1}^3 \nu_i \omega_j \p_j \omega_i dS+t\int_{\{|\bm{\omega}|=t\}}|{\bm f}|dS\\
&-\int_{\{|\bm{\omega}|=t\}}(\bm\omega\times\bm\nu)(G'(|\bm\omega|)\nabla|\bm\omega|\times\bm\omega)dS-t\int_{\{|\bm{\omega}|=t\}}G(|\bm\omega|)|\nabla|\bm\omega||dS.
\endaligned
\end{equation}
Furthermore, direct computations show that
$$
\aligned
&-\int_{\{|\bm{\omega}|=t\}}(\bm\omega\times\bm\nu)(G'(|\bm\omega|)\nabla|\bm\omega|\times\bm\omega)dS-t\int_{\{|\bm{\omega}|=t\}}G(|\bm\omega|)|\nabla|\bm\omega||dS\\
=&-(p-1)t^{p-1}\int_{\{|\bm{\omega}|=t\}}|\nabla|\bm\omega||dS.
\endaligned
$$
Consequently,
\begin{equation}\label{1.12}
\aligned
&\int_{\p\{|\bm{\omega}|>t\}}[G(|\bm\omega|)\curl\bm\omega\times\bm\omega+G(|\bm\omega|)\nabla|\bm\omega|\bm\omega]\cdot\bm\nu dS\\
\leq&\int_{\p\O\bigcap \p\{|\bm{\omega}|>t\}} G(|\bm{\omega}|)\sum_{i,j=1}^3 \nu_i \omega_j \p_j \omega_i dS+t\int_{\{|\bm{\omega}|=t\}}|{\bm f}|dS\\
&-(p-1)t^{p-1}\int_{\{|\bm{\omega}|=t\}}|\nabla|\bm\omega||dS.
\endaligned
\end{equation}
Combining \eqref{1.9} and \eqref{1.12}, we obtain \eqref{1.5}.

Now we consider the following two cases separately: (A) $p\geq2$, and (B) $p< 2$.

In case (A), multiplying both sides of \eqref{1.5} by $t^{p-2}(\int_{\{|\bm{\omega}|=t\}}|\nabla|\bm\omega||dS)^{-1}$ and then integrating the resulting inequality from $t_0$ to $T$, we have
\begin{equation}\label{1.13}
\aligned
\int_{t0}^T(p-1)t^{2p-3}dt\leq&\int_{t0}^T(\int_{\{|\bm{\omega}|=t\}}|\nabla|\bm\omega||dS)^{-1}\int_{\{|\bm{\omega}|>t\}}\frac{1}{4(p-1)}|{\bm f}|^2dxdt\\
&+\int_{t0}^T(\int_{\{|\bm{\omega}|=t\}}|\nabla|\bm\omega||dS)^{-1}t^{p-1}\int_{\{|\bm{\omega}|=t\}}|{\bm f}|dSdt\\
&+\int_{t0}^T(\int_{\{|\bm{\omega}|=t\}}|\nabla|\bm\omega||dS)^{-1}t^{p-2}\int_{\p\O\bigcap \p\{|\bm{\omega}|>t\}} G(|\bm{\omega}|)
\sum_{i,j=1}^3 \nu_i \omega_j \p_j \omega_i dSdt.
\endaligned
\end{equation}

The estimates for the first and second integrals in the right-hand side of \eqref{1.13} can be achieved directly from \cite[(2.16) (2.17)]{xiangxf}, in which infers that
\begin{equation}\label{1.14}
\aligned
&\int_{t_0}^{T}(\int_{\{|\bm{\omega}|=t\}}|\nabla|\nabla \bm\omega||dS)^{-1}(\int_{\{|\bm{\omega}|>t\}}\frac{1}{4(p-1)}|{\bm f}|^2dx)dt \leq \frac{1}{4(p-1)}C(\O) \|{\bm f}\|_{L^{3,1}(\O)}^2,
\endaligned
\end{equation}
and
\begin{equation}\label{1.15}
\aligned
\int_{t_0}^{T}(\int_{\{|\bm{\omega}|=t\}}|\nabla|\nabla\bm\omega||dS)^{-1}(\int_{\{|\bm{\omega}|>t\}}t^{p-1}|{\bm f}|dx)dt \leq T^{p-1}C(\O) \|{\bm f}\|_{L^{3,1}(\O)},
\endaligned
\end{equation}
where the constant C depends only on the size of domain $\O$ and the regularity of the domain (the cone condition for $\O$).

Moreover, the last integral on the right-hand side of \eqref{1.13} is negative. Indeed, according to \cite[p.135-137]{Grisvard1985Elliptic}, it follows that $$
\sum_{i,j=1}^3 \nu_i \omega_j \p_j \omega_i
=2\bm{\omega}_{\tau}\cdot\nabla_{\tau}\omega_n
-\mathrm{Div}(\omega_n \bm{\omega}_{\tau} )+
\mathscr{B}(\bm{\omega}_{\tau},\bm{\omega}_{\tau})+
\mathrm{tr}\mathscr{B} \omega_n^2.
$$
Since $\omega_n=0$ on $\p\O$ and the domain $\O$ is convex, for any $x\in\p\O$, we have
$$\sum_{i,j=1}^3 \nu_i \omega_j \p_j \omega_i
=\mathscr{B}(\bm{\omega}_{\tau},\bm{\omega}_{\tau}) \leq 0.$$

Note that
$$
\aligned
t_0:=|\bm\omega|^*(|\O|/2)&\leq C(p, \O)\|\curl \u\|_{L^p(\O)}\\
&\leq C(p, \O)\|{\bm f}\|^{\frac{1}{p-1}}_{L^{3,1}(\O)},
\endaligned
$$
where the last inequality follows from \eqref{7.3}. Substituting \eqref{1.14}, \eqref{1.15} into \eqref{1.13}, we have
\begin{equation}\label{1.17}
CT^{2p-2}
\leq C(p, \O)\lVert {\bm f}\rVert ^2_{L^{3,1}(\O)}+C(\O)T^{p-1}\lVert {\bm f}\rVert _{L^{3,1}(\O)}.
\end{equation}
Now letting $T\rightarrow \lVert\bm\omega\rVert _{L^{\infty}(\O)}$, we consequently obtain that for any $x\in\O$, there exists a constant $C$ depending on $p,\O$, such that
\begin{equation}\label{9.15}
\lVert\curl\u\rVert _{L^{\infty}(\O)}\leq C \|{\bm f}\|_{L^{3,1}(\O)}^{\frac{1}{p-1}}.
\end{equation}

In case (B), multiplying both sides of \eqref{1.5} by $(\int_{\{|\bm{\omega}|=t\}}|\nabla|\bm\omega||dS)^{-1}$ and then integrating the resulting inequality from $t_0$ to $T$, from \eqref{1.5} we obtain that
\begin{equation}
\aligned
\int_{t_0}^{T}(p-1)t^{p-1}dt\leq&\int_{t_0}^{T}(\int_{\{|\bm{\omega}|=t\}}|\nabla|\bm\omega||dS)^{-1}\int_{\{|\bm{\omega}|>t\}}\frac{1}{4(p-1)}T^{2-p}|{\bm f}|^2dxdt\\
&+\int_{t_0}^{T}(\int_{\{|\bm{\omega}|=t\}}|\nabla|\bm\omega||dS)^{-1}T\int_{\{|\bm{\omega}|=t\}}|{\bm f}|dSdt\\
&+\int_{t_0}^{T}(\int_{\{|\bm{\omega}|=t\}}|\nabla|\bm\omega||dS)^{-1}\int_{\p\O\bigcap \p\{|\bm{\omega}|>t\}} G(|\bm{\omega}|)
\sum_{i,j=1}^3 \nu_i \omega_j \p_j \omega_i dSdt.
\endaligned,
\end{equation}
Similarly, we have
$$CT^p\leq C(p,\O)T^{2-p}\|{\bm f}\|_{L^{3,1}(\O)}^2+C(\O)T\|{\bm f}\|_{L^{3,1}(\O)}+C(p,\O)\|{\bm f}\|_{L^{3,1}(\O)}^{\frac{p}{p-1}}.$$
Letting $T\rightarrow \lVert\bm\omega\rVert _{L^{\infty}(\O)}$, we consequently obtain that for any $x\in\O$, there exists a constant $C$ depending on $p,\O$, the estimate \eqref{9.15} is established.

Step 2. We remove the assumption \eqref{2.3}. We look
for a sequence $\{\O_m\}_{m\in \mathbb{N}}$
of bounded domains $\O_m\supset \O$ such that
$\O_m\in C^{\infty}, |\O_m\backslash \O|\to 0, \O_m\to\O$
with respect to the Hausdorff distance (see \cite{Verchota1984}). Let $\u_m$
be
the solutions to the system \eqref{1.1}
with the domain $\O$ replaced by $\O_m$. From step 1,
 we see that
$$
\|\curl\u_m\|_{L^{\infty}(\O_m)}\leq C \|{\bm f}\|_{L^{3,1}(\O_m)}^{\frac{1}{p-1}},
$$
where the constant $C$ depends on $p$ and $\O$ (the Lipschitz character of the domain).
Based on this fact, then by the $C^{1,\alpha}$ local regularity theory,
for any compact subset $\mathcal {K}$ of $\O$, we have
$$
\|\u_m\|_{C^{1,\alpha}(\mathcal {K})}\leq C \left(\mathcal {K}, \O\right),
$$
see \cite{Lieberman1988}.
By Arezela's theorem, we see that
$$
\curl\u_m\to \curl\hat{\u} \q\q a.e. \q \text{in }\O.
$$
For any $\psi\in W_t^p(\O,\dv 0)$,
we have
$$
\aligned
&\int_{\O}|\curl\u_m|^{p-2}\curl\u_m \cdot \curl\psi dx
=\int_{\O}{\bm f}\cdot\psi dx.
\endaligned
$$
Now by Lebesgue's dominated convergence theorem, we have
$$
\aligned
&\int_{\O}|\curl\u|^{p-2}\curl\u\cdot \curl\psi dx
=\int_{\O}{\bm f}\cdot\psi dx.
\endaligned
$$

Step 3. We remove the assumption \eqref{2.4}. We first extend the vector ${\bm f}$ to $\O_0$ with
$\O\subset\O_0$. There exists an extension of ${\bm f}$, denoted by $\tilde{{\bm f}}$,
such that $\dv\tilde{{\bm f}}=0$ in $\O_0$ and satisfies
$$
\|\tilde{{\bm f}}\|_{L^{3,1}(\O_0)}\leq C\|\tilde{{\bm f}}\|_{L^{3,1}(\O)}.
$$
There exists a sequence $\tilde{{\bm f}}_n\in C^{\infty}(\O_0)$ with $\dv\tilde{{\bm f}_n}=0$ and
$$
\tilde{{\bm f}}_n\to \tilde{{\bm f}} \q \text{in } L^{3,1}(\O_0),
$$
for the proof we may refer to \cite{Costabel1990A}.
The following proof is the same as step 4 in the proof of Theorem 1.1 in \cite{Cianchi2011}.
The proof is complete.
\end{proof}
\v0.2in

\noindent{\bf ACKNOWLEDGMENTS}
\v0.1in
The author wishes to thank her supervisor Professor Baojun Bian for his persistent guidance and constant encouragement. The research work was partly supported by the National Natural Science Foundation of China grant No. 11771335.


\vspace {0.5cm}

\begin {thebibliography}{DUMA}
\bibitem{Antontsev}
S. Antontsev, F. Miranda, L. Santos,
\newblock A class of electromagnetic p-curl systems: blow-up and finite time extinction,
\newblock {Nonlinear Anal.} 75 (2012) 3916-3929.

\bibitem{FAntontsev}
S. Antontsev, F. Miranda, L. Santos,
\newblock Blow-up and finite time extinction for p(x,t)-curl systems arising in electromagnetism,
\newblock {J. Math. Anal. Appl.} 440 (2016) 300-322.

\bibitem{Aramaki}
J. Aramaki,
\newblock On a Degenerate evolution system associated with the Bean critical-state for type II superconductors,
\newblock {Abstr. Appl. Anal.} (2015) 1-13.

\bibitem{Barrett1993Finite}
J.W. Barrett, W.B. Liu,
\newblock Finite element approximation of the p-laplacian,
\newblock {Mathematics of Computation.} 61 (204)(1994) 523-537.

\bibitem{Barrett}
J.W. Barrett, L. Prigozhin,
\newblock Bean’s critical-state model as the $p\to\infty$ limit of an evolutionary p-Laplacian equation,
\newblock{Nonlinear Anal.} 42 (2000) 977-993.

\bibitem{Bartsch}
T. Bartsch, J. Mederski,
\newblock Ground and bound state solutions of semilinear time-harmonic Maxwell equations in a bounded domain,
\newblock {Arch. Ration. Mech. Anal.} 215 (2014) 283-306.

\bibitem{Bean1964Magnetization}
C.P. Bean,
\newblock Magnetization of high-field superconductors,
\newblock {Rev. mod. phys.} 36 (1)(1964) 886-901.

\bibitem{Chapman2000A}
S.J. Chapman,
\newblock A hierarchy of models for type-II superconductors,
\newblock {Siam Review.} 42 (4)(2000) 555-598.

\bibitem{ChenJun}
J. Chen, X.B. Pan,
\newblock Quasilinear systems involving curl,
\newblock {Proc. Royal Soc. Edingburgh Sect. A Mathematics } 148 (2)(2018) 1-37.

\bibitem{Cianchi2014Global}
A. Cianchi, V.G. Maz'ya,
\newblock Global boundedness of the gradient for a class of nonlinear elliptic
systems,
\newblock {Arch. Ration. Mech. Anal.} 212 (2014) 129-177.

\bibitem{Cianchi2011}
A. Cianchi, V.G. Maz'ya,
\newblock Global Lipschitz regularity for a class of quasilinear elliptic
equations,
\newblock{Comm. Part. Differ. Eq.} 36 (2011) 100-133.

\bibitem{Costabel1990A}
M. Costabel,
\newblock A remark on the regularity of solutions of Maxwell's equations on
Lipschitz domains,
\newblock {Math. Method Appl Sci.} 12 (4)(1990) 365-368.

\bibitem{DeGennes}
P.G. DeGennes,
\newblock Superconducting of Metal and Alloys,
\newblock {Benjamin, New York,} 1966.

\bibitem{Dibenedetto2010Degenerate}
E. Dibenedetto,
\newblock Degenerate parabolic equations,
\newblock {Springer-Verlag.} 2010.

\bibitem{Grisvard1985Elliptic}
P. Grisvard,
\newblock Elliptic problems in nonsmooth domains,
\newblock {Pitman Advanced Pub. Program.} 1985.

\bibitem{Laforest}
M. Laforest,
\newblock The p-curlcurl: Spaces, traces, coercivity and a Helmholtz decomposition in $L_p$,
\newblock preprint.

\bibitem{Lieberman1988}
G. Lieberman,
\newblock Boundary regularity for solutions of degenerate elliptic equations,
\newblock {Nonlinear Anal.} 12 (11)(1988) 1203-1219.

\bibitem{Mederski}
J. Mederski,
\newblock Ground states of time-harmonic Maxwell equations in $R^3$ with vanishing permittivity,
\newblock {Arch. Ration. Mech. Anal.} 218 (2)(2015) 825-861.

\bibitem{Str}
M. Struwe,
\newblock Variational Methods,
\newblock 3rd edition, Springer-Verlag, Berlin, 2006.

\bibitem{Tang}
X. Tang, D. Qin,
\newblock Ground state solutions for semilinear time-harmonic Maxwell equations,
\newblock {J. Math. Phys.} 57 (4)(2016) 823-864.

\bibitem{Uralceva}
N.N. Ural'ceva,
\newblock Degenerate quasilinear elliptic systems,
\newblock {Zap. Nauchn. Sem. Leningrad. Otdel. Mat. Inst. Steklov. (LOMI)} 7 (1968) 184-222 (Russian).

\bibitem{Verchota1984}
G. Verchota,
\newblock Layer potentials and regularity for the Dirichlet problem for Laplace's equation in Lipschitz domains,
\newblock {J. Funct. Anal.} 59 (3)(1984) 572-611.

\bibitem{xiangxf}
X.F. Xiang,
\newblock $L^{\infty}$ estimate for a limiting form of Ginzburg-Landau systems in
convex domains,
\newblock {J. Math. Anal. Appl.} 438 (1)(2016) 328-338.

\bibitem{XiangM}
M.Q. Xiang, F.L. Wang, B.L. Zhang,
\newblock Existence and multiplicity of solutions for p(x)-curl systems arising in electromagnetism,
\newblock {J. Math. Anal. Appl.} 448 (2)(2016) 1600-1617.

\bibitem{Yin1998On}
H.M. Yin,
\newblock On a singular limit problem for nonlinear Maxwell's equations,
\newblock {J. Diff. Eqs.} 156 (2)(1998) 355-375.

\bibitem{Yin2001}
H.M. Yin,
\newblock On a p-Laplacian type of evolution system and applications to the Bean model in
the type-II superconductivity theory,
\newblock {Quart. Appl. Math.} 59 (1)(2001) 47-66.


\bibitem{Yin2002}
H.M. Yin, B.Q. Li, J. Zou,
\newblock A degenerate evolution system modeling Bean's critical-state type-II superconductors,
\newblock {Discrete Contin. Dyn. Syst. Ser. A.} 8 (3)(2002) 781-794.

\bibitem{Yin2006Regularity}
H.M. Yin,
\newblock Regularity of weak solution to an p-curl system,
\newblock {Differ. Integral Equ.} 19 (4)(2006) 361-368.


\end{thebibliography}
\end {document}